\documentclass[10pt, reqno]{amsart}

\usepackage{amsmath,amssymb,amsthm,amsfonts,verbatim}
\usepackage{microtype}
\usepackage[all,2cell]{xy}
\usepackage{mathtools}
\usepackage{graphicx}
\usepackage{pinlabel}
\usepackage{hyperref}
\usepackage{mathrsfs}
\usepackage{color}
\usepackage{enumerate}
\usepackage{cite}

\usepackage[top=1.3in,bottom=1.9in,left=1.2in,right=1.2in]{geometry}

\theoremstyle{plain}
\newtheorem{theorem}{Theorem}[section]
\newtheorem{maintheorem}{Theorem}

\newtheorem{question}[theorem]{Question}

\newtheorem{proposition}[theorem]{Proposition}
\newtheorem{lemma}[theorem]{Lemma}

\theoremstyle{definition}
\newtheorem{definition}[theorem]{Definition}

\newtheorem{remark}[theorem]{Remark}

\newcommand{\nc}{\newcommand}
\nc{\dmo}{\DeclareMathOperator}

\nc{\Q}{\mathbb{Q}}
\nc{\F}{\mathbb{F}}
\nc{\R}{\mathbb{R}}
\nc{\Z}{\mathbb{Z}}
\nc{\C}{\mathbb{C}}
\nc{\Ell}{\mathcal{L}}
\nc{\M}{\mathcal{M}}
\nc{\K}{\mathcal{K}}
\nc{\I}{\mathcal{I}}
\nc{\U}{\mathcal U}
\nc{\disk}{\mathbb{D}}
\nc{\hyp}{\mathbb{H}}

\nc{\CP}{\mathbb{CP}}
\nc{\cS}{\mathcal{S}}
\dmo{\Mod}{Mod}
\dmo{\PMod}{PMod}
\dmo{\Diff}{Diff}
\dmo{\Homeo}{Homeo}
\dmo{\dist}{dist}
\dmo\BDiff{BDiff}
\dmo\SO{SO}
\dmo\Hom{Hom}
\dmo\SL{SL}
\dmo\Sp{Sp}
\dmo\rank{rank}
\dmo\sig{sig}
\dmo\Out{Out}
\dmo\Aut{Aut}
\dmo\Inn{Inn}
\dmo\GL{GL}
\dmo\PSL{PSL}
\dmo\BHomeo{BHomeo}
\dmo\EHomeo{EHomeo}
\dmo\EDiff{EDiff}
\nc\Sig{\Sigma}
\dmo\Teich{Teich}
\dmo\Fix{Fix}
\nc{\pair}[1]{\langle #1 \rangle}
\nc{\abs}[1]{\left| #1 \right|}
\nc{\action}{\circlearrowright}
\nc{\norm}[1]{\left | \left | #1 \right | \right |}
\nc{\abcd}[4]{\left(\begin{array}{cc} #1 & #2 \\ #3 & #4 \end{array}\right)}
\nc{\into}\hookrightarrow
\dmo{\Isom}{Isom}
\nc{\normal}{\vartriangleleft}
\dmo{\Vol}{Vol}
\dmo{\im}{Im}
\dmo{\Push}{Push}
\dmo{\Conf}{Conf}
\dmo{\PConf}{PConf}
\dmo{\id}{id}
\dmo{\Jac}{Jac}
\dmo{\Pic}{Pic}
\dmo{\Stab}{Stab}
\dmo{\Arf}{Arf}
\dmo{\End}{End}
\dmo{\PB}{PB}
\dmo{\CRS}{CRS}
\dmo{\PGL}{PGL}
\dmo{\Sym}{Sym}
\dmo{\Cov}{Cov}
\nc{\Span}[1]{\operatorname{Span}(#1)}

\renewcommand{\epsilon}{\varepsilon}
\renewcommand{\tilde}{\widetilde}
\renewcommand{\bar}{\overline}
\renewcommand{\int}{\operatorname{Int}}
\nc{\coloneq}{\mathrel{\mathop:}\mkern-1.2mu=}
\nc{\margin}[1]{\marginpar{\scriptsize #1}}
\nc{\para}[1]{\medskip\noindent\textbf{#1.}}
\nc{\red}[1]{\textcolor{red}{#1}}

\title{Section problems for configurations of points on the Riemann sphere}

\author{Lei Chen \and Nick Salter}
\email{chenlei1991919@gmail.com \and salter.n@gmail.com}
\address{LC: Department of Mathematics, Caltech, Pasadena, CA; NS: Department of Mathematics, Columbia University, New York, NY}
\thanks{N.S. gratefully acknowledges support by the National Science Foundation under Award No. DMS-1703181.}

\date{July 24, 2018}

\begin{document}
\maketitle

\begin{abstract}
This paper contains a suite of results concerning the problem of adding $m$ distinct new points to a configuration of $n$ distinct points on the Riemann sphere, such that the new points depend continuously on the old. Altogether, the results of the paper provide a complete answer to the following question: given $n \ne 5$, for which $m$ can one continuously add $m$ points to a configuration of $n$ points? For $n \ge 6$, we find that $m$ must be divisible by $n(n-1)(n-2)$, and we provide a construction based on the idea of cabling of braids. For $n = 3,4$, we give some exceptional constructions based on the theory of elliptic curves. 
\end{abstract}

\section{Introduction}
This paper studies the space $\Conf_n(S^2)$ of configurations of $n$ distinct unordered points in $S^2$. This is the base space for a fiber bundle $P: \Conf_{n,m}(S^2) \to \Conf_n(S^2)$, where the total space $\Conf_{n,m}(S^2)$ is the space of configurations of $n+m$ distinct points divided into two groups of cardinalities $n$ and $m$. For any fiber bundle $\pi: E \to B$, it is a basic question to understand the space of {\em sections}, i.e. continuous maps $\sigma: B \to E$ satisfying $\pi \circ \sigma = \id$. In the case of $P: \Conf_{n,m}(S^2) \to \Conf_n(S^2)$, a section $S: \Conf_n(S^2) \to \Conf_{n,m}(S^2)$ has a very natural interpretation: $S$ is an assignment of $m$ additional points to a given configuration of $n$ distinct points that depends continuously on the position of the $n$ points. 

The approach we pursue in this paper is to study sections of $P$ by means of the fundamental group. The {\em spherical braid group} $B_n(S^2)$ is the fundamental group of $\Conf_n(S^2)$, and we also define 
\[
B_{n,m}(S^2) := \pi_1(\Conf_{n,m}(S^2)).
\]
Setting $p:=P_*$ and $s:= S_*$, a section $S$ induces a group-theoretic section $s: B_n(S^2) \to B_{n,m}(S^2)$ of the surjective homomorphism $p: B_{n,m}(S^2) \to B_n(S^2)$. Thus an obstruction to the existence of a group-theoretic section $s$ furnishes an obstruction to the existence of a bundle-theoretic section $S$. A standard argument in obstruction theory shows that the converse is true as well.

The theory of sections of bundles of configuration and moduli spaces plays an important role in topology, geometric group theory, and algebraic geometry. See, for instance, the work of L. Chen \cite{lei1}, \cite{lei2}, \cite{lei3}, L. Chen--Salter \cite{leinick}, W. Chen \cite{weiyan},  Lin \cite{lin} or the classic papers of Earle--Kra \cite{earlekra1}, \cite{earlekra2} and Hubbard \cite{hubbard}, each of which treats various instances of the problem of obstructing and classifying sections of such bundles. 

The section problem for $P: \Conf_{n,m}(S^2) \to \Conf_n(S^2)$ is particularly subtle and rich for several reasons. If the ambient space $S^2$ is replaced with $\C$, sections of $\Conf_{n,m}(\C) \to \Conf_n(\C)$ are easy to construct: one can simply add $m$ new points ``near infinity''. By contrast, even the mere {\em existence} of sections in the spherical case is far from obvious. Foundational work on this question was carried out by Gon\c calves-Guaschi in \cite{GG}. They established the following intriguing theorem.

\begin{theorem}[Gon\c calves-Guaschi]\label{theorem:GG}
A group-theoretic section $s: B_3(S^2) \to B_{3,m}(S^2)$ exists if and only if $m \equiv 0$ or $m \equiv 2 \pmod 3$.

For $n \ge 4$, there are no sections $s: B_n(S^2) \to B_{n,m}(S^2)$ except possibly if $m$ is congruent to one of the four residues $0,(n-1)(n-2), -n(n-2), -(n-2)$ mod $n(n-1)(n-2)$.
\end{theorem}
Gon\c calves-Guaschi did not give any explicit construction of sections, even in the $n = 3$ case. Our first main theorem addresses the case $m \equiv 0 \pmod{n(n-1)(n-2)}$, giving a construction of a family of sections based on the idea of ``cabling'' of braids.
\begin{maintheorem}\label{theorem:cabling}
For any $n \ge 3$ and any $m$ divisible by $n(n-1)(n-2)$, there is a section $S: \Conf_n(S^2) \to \Conf_{n,m}(S^2)$. 
\end{maintheorem}

By virtue of the identification $S^2 \cong \CP^1$, the section problem on $S^2$ has deep connections with algebraic geometry; indeed $\Conf_{n}(\CP^1)$ is a quasi-projective algebraic variety. Accordingly, we will switch between the equivalent notations $\Conf_n(S^2)$ and $\Conf_n(\CP^1)$ as the situation dictates. The algebro-geometric perspective is exploited in Section \ref{section:34}, where we study the special cases $n = 3$ and $n = 4$. We construct some exceptional sections by making use of the cross-ratio for $n=3$ and the theory of elliptic curves for $n = 4$. Note that for $n = 4$, the residues appearing in Theorem \ref{theorem:GG} are $0,6,16,22$ mod $24$. 

\begin{maintheorem}\label{theorem:34}
For any $m \ge 0$ satisfying $m \equiv 0,2 \pmod 3$, there exists an algebraic map $S$ that gives a section of the bundle $\Conf_{3,m}(\CP^1) \to \Conf_3(\CP^1)$.

For any $m \ge 70$ such that $m$ is congruent to one of $0,6,16,22$ mod $24$, there is a section $S: \Conf_4(\CP^1) \to \Conf_{4,m}(\CP^1)$.
\end{maintheorem}
The sections constructed for $m = 4$ are a hybrid of algebraic and non-algebraic sections; see Section \ref{section:4} for more details.

Finally, we address the case of $n \ge 6$. Here we are able to give a substantial strengthening of Theorem \ref{theorem:GG}, which, in combination with Theorem \ref{theorem:cabling}, gives a complete determination of those $m$ for which there exist sections $S: \Conf_n(S^2) \to \Conf_{n,m}(S^2)$. While our proof of Theorem \ref{theorem:main} is logically independent from Theorem \ref{theorem:GG}, we are indebted to the work of Gon\c calves-Guaschi both for some insights that we build off of, and moreover for the formulation of such a tantalizing problem. 

\begin{maintheorem}\label{theorem:main}
For $n \ge 6$, no section $S: \Conf_n(S^2) \to \Conf_{n,m}(S^2)$ exists unless $n(n-1)(n-2)$ divides $m$.
\end{maintheorem}

It is natural to ask the extent to which these results are ``predicted'' by algebraic geometry. One way of making this precise is to ask whether every section $s: \Conf_n(S^2) \to \Conf_{n,m}(S^2)$ is homotopic to a map of varieties $s': \Conf_n(\CP^1) \to \Conf_{n,m}(\CP^1)$. Perhaps surprisingly, it turns out that for $n \ge 5$ the projection map $p: \Conf_{n.m}(\CP^1) \to \Conf_n(\CP^1)$ {\em has no section given by an algebraic map}. This is an instance of a theorem of Lin \cite[Theorem 3]{lin}, which builds off of the work of Earle--Kra (see \cite[Section 4.6]{earlekra1}). Thus the problem of constructing and classifying {\em continuous} sections of $\Conf_{n,m}(S^2) \to \Conf_n(S^2)$ is genuinely different from the analogous problem in the algebraic category. 

The results of this paper only concern the (non)existence of sections of $\Conf_{n,m}(S^2) \to \Conf_n(S^2)$. We have not addressed the question of {\em uniqueness}, but we believe that this is worthy of further study.

\begin{question}
Is every section $S: \Conf_n(S^2) \to \Conf_{n,m}(S^2)$ homotopic to one of the sections constructed in Theorems \ref{theorem:cabling} or Theorem \ref{theorem:34}?
\end{question}

Our method of proof for Theorem \ref{theorem:main} is to exploit the close relationship between $B_n(S^2)$ and the mapping class group $\Mod_n(S^2)$ of the $n$-punctured sphere. The analogous section problem for mapping class groups is amenable to the powerful theory of {\em canonical reduction systems}. We prove Theorem \ref{theorem:main} by analyzing the canonical reduction systems for a particularly convenient generating set of $\Mod_n(S^2)$. It is perhaps surprising that this approach totally avoids any analysis of the eponymous braid relation! Rather, building off of some observations by Gon\c calves-Guaschi, a key role is played by the torsion elements of $\Mod_n(S^2)$. 

\para{Contents of the paper} Section \ref{section:sphericalbraid} recalls some basic facts about the groups $B_n(S^2)$ and $\Mod_n(S^2)$. Section \ref{section:cable} discusses ``cabling'' of braids and exploits this to give the construction from which Theorem \ref{theorem:cabling} follows. In Section \ref{section:34}, we give some algebro-geometric constructions of sections in the cases $n = 3,4$, establishing Theorem \ref{theorem:34}. 

The proof of Theorem \ref{theorem:main} is carried out in Sections \ref{section:CRS} - \ref{section:B}. In Section \ref{section:CRS}, we review the theory of canonical reduction systems. In Section \ref{section:prelims} we establish some preliminary notions, leading to an overview of the proof given in Section \ref{section:overview}. The proof itself is carried out in Sections \ref{section:reducible}-\ref{section:B}.

\para{Acknowledgements} The authors would like to thank Ian Frankel for some helpful suggestions concerning Theorem \ref{theorem:34}, and Dmitri Gekhtman for some insights in {T}eichm\"uller theory. 

\section{The spherical braid group}\label{section:sphericalbraid}
In this section, we remind the reader of the relevant aspects of the theory of the spherical braid groups and their relationship with the mapping class groups of the punctured sphere.

\subsection{The (spherical) braid group} Let $S$ be a surface of finite type and let $\PConf_k(S)$ denote the space of {\em ordered} $k$-tuples of distinct points on $S$. The symmetric group $S_k$ acts on $\PConf_k(S)$ by permuting the ordering of the points; this action is by deck transformations. For any subgroup $G \le S_k$, there is an associated covering space $\PConf_k(S) \to \PConf_k(S)/G$. 

In the case $S = D^2$ the open disk and $G = S_k$, the space $\Conf_k(D^2) := \PConf_k(D^2)/S_k$ has fundamental group given by the classical braid group
\[
B_k := \pi_1(\Conf_k(D^2)).
\]

The primary surface of interest in this paper is the Riemann sphere $S = S^2$. There are two subgroups $G$ as above that will be of interest. First is $G = S_k$. We write
\[
\Conf_k(S^2) := \PConf_k(S^2)/ S_k;
\]
this is the space of {\em unordered} $k$-tuples of distinct points on $S^2$. The {\em spherical braid group} is defined to be the fundamental group of this space:
\[
B_k(S^2) := \pi_1(\Conf_k(S^2)).
\]
Secondly, suppose $k = n+m$, and consider the subgroup $G = S_n \times S_m$ of $S_k$. We write
\[
\Conf_{n,m}(S^2) := \PConf_k(S^2) / (S_n \times S_m).
\]
This can be viewed as the space of $n$ ``red'' points and $m$ ``blue'' points which are otherwise indistinguishable. This leads to a useful piece of terminology.

\begin{definition}[Old points, new points]\label{definition:oldnew}
Relative to the preceding discussion, we refer to the set of cardinality $n$ as the set of {\em old points}, and the set of cardinality $m$ as the set of {\em new points}. The set of old points is written $\{x_1, \dots, x_n\}$, and the set of new points is written $\{y_1, \dots, y_m\}$. 
\end{definition}

We define
\[
B_{n,m}(S^2) := \pi_1(\Conf_{n,m}(S^2)).
\]
There is an evident forgetful map $P: \Conf_{n,m}(S^2) \to \Conf_n(S^2)$ giving rise to a {\em surjective} homomorphism
\[
p: B_{n,m}(S^2) \to B_n(S^2).
\]
It is this $p$ that we seek to find (obstructions to) sections of. Recall that a {\em section} of a surjective group homomorphism $p: A \to B$ is a (necessarily injective) homomorphism $s: B \to A$ satisfying $p \circ s = \id$.

\subsection{A presentation of the spherical braid group} It is classically known that $B_n(S^2)$ has a presentation obtained by adding a single relation to Artin's presentation of the classical braid group $B_n$. Let $R_n$ be the word in $\sigma_1, \dots, \sigma_{n-1}$ given by
\begin{equation}\label{equation:Rn}
R_n = \sigma_1 \dots \sigma_{n-1} \sigma_{n-1} \dots \sigma_1.
\end{equation}
Then $B_n(S^2)$ has the following presentation:
\begin{align*}
B_n(S^2) = \langle \sigma_1, \dots, \sigma_{n-1} \mid	& [\sigma_i, \sigma_j]=1 \mbox{ for } \abs{i-j} > 1,\\
											&\sigma_i \sigma_{i+1} \sigma_{i} = \sigma_{i+1} \sigma_{i} \sigma_{i+1} \mbox{ for }1\le i \le n-2,\\
											& R_n =1 \rangle.
\end{align*}

The element $R_n$ is depicted in Figure \ref{figure:rn}.
\begin{figure}
\labellist
\small
\endlabellist
\includegraphics[scale=0.8]{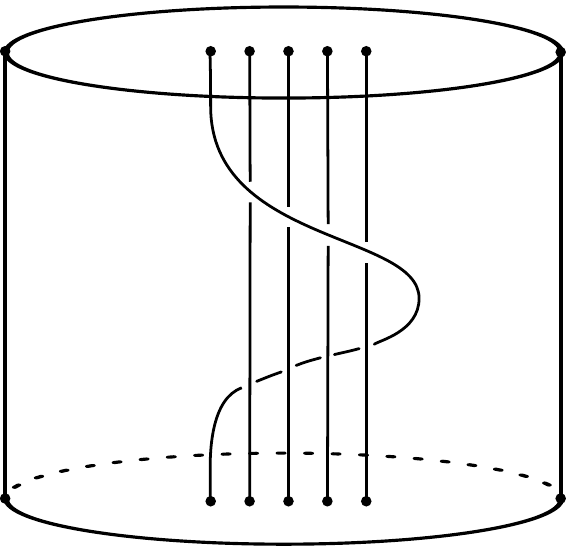}
\caption{The braid $R_n = \sigma_1 \dots \sigma_{n-1}^2 \dots \sigma_1 \subset D^2 \times [0,1]$.}
\label{figure:rn}
\end{figure}

\subsection{Torsion in the spherical braid group} In \cite{murasugi}, Murasugi determined the finite-order elements in $B_n(S^2)$. He showed that every finite-order element is conjugate to a power of one of the following three elements. Their properties are summarized in the table below. 
\begin{equation}\label{table:torsion}
\begin{array}{|c|ccccc|}
\hline
\mbox{Element}	& \mbox{Expression}						&&\mbox{Order}	&& \mbox{Permutation}\\ \hline
\alpha_0 			& \sigma_1 \dots \sigma_{n-1}				&&2n			&& (1\dots n)		\\
\alpha_1 			& \sigma_1 \dots \sigma_{n-2}\sigma_{n-1}^2	&&2n-2			&& (1 \dots n-1)		\\
\alpha_2 			& \sigma_1 \dots \sigma_{n-3} \sigma_{n-2}^2	&&2n-4			&& (1 \dots n-2)		\\ \hline
\end{array}
\end{equation}

\begin{remark}\label{remark:powers}
From the table, one can easily determine the permutation associated to any {\em power} of (a conjugate of) $\alpha_i$. Explicitly, $\alpha_i^k$ has permutation given by $(1 \dots n-i)^k$, which decomposes into $\gcd(k, n-i)$ disjoint $j$-cycles, where $j = (n-i) / \gcd(k, n-i)$.
\end{remark}

\subsection{An alternative generating set}
Although the set $\{\sigma_1, \dots, \sigma_{n-1}\}$ of ``standard'' generators for $B_n$ is the most widely-known, it is not the most useful for our purposes. The starting point for this discussion is the following elementary lemma whose proof is a direct calculation.

\begin{lemma} \label{lemma:alpha}
For $1 \le i \le n-2$, there is an equality
\[
\alpha_0^i \sigma_1 \alpha_0^{-i} = \sigma_{1+i}.
\]
of elements of $B_n$. Similarly, for $1 \le i \le n-3$, there is an equality
\[
\alpha_1^i \sigma_1 \alpha_1^{-i} = \sigma_{1+i}.
\]
\end{lemma}

As an immediate corollary, we obtain our desired generating set.
\begin{lemma}\label{lemma:genset}
$B_n$, and hence its quotient $B_n(S^2)$, is generated by the set $\{\sigma_1, \alpha_0\}$.
\end{lemma}

We will also have occasion to study the subgroup generated by the elements $\alpha_1$ and $\sigma_1$. This group admits the following convenient description.
\begin{lemma}\label{lemma:genset2}
Let $B_{n-1,1}(S^2) \leqslant B_n(S^2)$ be the subgroup consisting of braids that fix the point $x_n$. Then $B_{n-1,1}(S^2)$ is generated by the set $\{\sigma_1, \alpha_1\}$. 
\end{lemma}

\begin{proof}
In light of Lemma \ref{lemma:alpha}, it suffices to show that $\sigma_1,...,\sigma_{n-2}$ generates $B_{n-1,1}(S^2)$. Let $P_n(S^2)$ be the kernel of the map $p: B_n(S^2)\to S_n$ recording the permutation of points. There is a short exact sequence
\[
1\to P_{n}(S^2)\to B_{n-1,1}(S^2)\xrightarrow{\pi} S_{n-1}\to 1.
\]
As $\{\pi(\sigma_1),...,\pi(\sigma_{n-2})\}$ generates $S_{n-1}$, it suffices to show that the kernel $P_n(S^2)$ is contained in the subgroup $H = \pair{\sigma_1,...,\sigma_{n-2}}$ of $B_{n-1,1}(S^2)$. Since $R_n=1\in B_n(S^2)$, it follows that $\sigma_{n-1}^2\in H$. Let $P_n$ be the kernel of  the map $B_n\to S_n$ recording the permutation of points. We define the subgroup $G \leqslant B_n$ by
\[
G:= \pair{\sigma_1,...,\sigma_{n-2},\sigma_{n-1}^2}.
\]
Let $q: B_n \to B_n(S^2)$ be the natural projection from the classical braid group to the spherical braid group. By definition $q(G)=H$. For $1 \le i < j \le n-1$, we define
\[A_{i,j}=\sigma_i^{-1}\sigma_{i+1}^{-1}...\sigma_{j-2}^{-1}\sigma_{j-1}^2\sigma_{j-2}...\sigma_i.
\]
According to Artin \cite[Theorem 17]{artin}, the set $\{A_{i,j}\}$ generates $P_n$, and evidently $\{A_{i,j}\}\subset G$, so that $P_n \leqslant G$. Therefore we have that $q(P_n)\leqslant H$. Since we know that $q(P_n)=P_n(S^2)$, we get that $P_n(S^2)\leqslant H$.
\end{proof}

\subsection{From braid groups to mapping class groups}\label{section:braidtomod} Let $S$ be a surface and let $\Mod_n(S)$ denote the mapping class group of $S$ relative to $n$ unordered marked points (equivalently the marked points can be viewed as punctures). We also define $\Mod_{n,m}(S)$ as the subgroup of $\Mod_{n+m}(S)$ consisting of mapping classes that preserve a partitioning of the marked points into two sets of cardinalities $n,m$ respectively. 

The ``point-pushing construction'' (see, e.g. \cite[Section 9.1.4]{FM}) yields a homomorphism
\[
\mathcal P: B_n(S) \to \Mod_n(S).
\]
For most surfaces, $\mathcal P$ is an isomorphism onto its image, but this is not the case for $S = S^2$. Rather, there is the following short exact sequence (again, see \cite[Section 9.1.4]{FM}):
\[
\xymatrix{
1 \ar[r]	& \Z/2\Z \ar[r] &	 B_n(S^2) \ar[r]^-{\mathcal P}& \Mod_n(S^2) \ar[r]& 1.
}
\] 
The nontrivial element of the kernel is the central element $\omega$. It is characterized by the property that it is the unique element of order $2$ in $B_n(S^2)$. 

\begin{lemma}\label{lemma:2torsion}
There is a unique element $\omega \in B_n(S^2)$ of order $2$. 
\end{lemma}

\begin{proof}
Define $\omega := (\sigma_1 \dots \sigma_{n-1})^n$. This determines a central element of $B_n$, hence $\omega$ is central in $B_n(S^2)$ as well. By Murasugi's classification of torsion elements (Table \eqref{table:torsion}), any element $\beta$ of order $2$ is conjugate to $\alpha_{i}^{n-i}$ for some $i = 0,1,2$. The element $\alpha_0^n$ is equal to $\omega$. As $\omega$ is central, any conjugate of $\alpha_0^n$ is equal to $\omega$. It is straightforward to verify that the equality
\[
\alpha_1^{n-1} = \alpha_0^n = \omega
\]
holds in $B_n$, hence any conjugate of $\alpha_1^{n-1}$ is also equal to $\omega$. Finally, a direct (albeit tedious) calculation establishes the equality
\[
(\sigma_1 \dots \sigma_{n-1}) \alpha_2^{n-2} (\sigma_{n-1} \dots \sigma_1) = \omega
\]
in $B_n$. The elements $(\sigma_1 \dots \sigma_{n-1}) \alpha_2^{n-2} (\sigma_{n-1} \dots \sigma_1)$ and $\alpha_2^{n-2}$ are equal in $B_n(S^2)$, so that any conjugate of $\alpha_2^{n-2}$ is equal to $\omega$ as well.
\end{proof}

The classification of torsion in $B_n(S^2)$ given in \eqref{table:torsion} ports directly over to $\Mod_n(S)$. The only difference between torsion in the mapping class group is that the element $\alpha_i$ has order $n-i$ as opposed to $2n-2i$. As mapping classes, the torsion elements $\alpha_i$ have simple geometric representatives. Arrange the points $x_1, \dots, x_{n-1-i}$ at equal intervals on the equator, and place any remaining points at the north and south poles. As a mapping class, $\alpha_i$ is then represented by a rotation of the sphere by an angle of $2\pi/(n-i)$ through the plane of the equator. \\

It is {\em a priori} possible that a section $s: B_n(S^2) \to B_{n,m}(S^2)$ could fail to descend to a section $\bar s: \Mod_n(S^2) \to \Mod_{n,m}(S^2)$. If this were the case, we would not be able to prove Theorem \ref{theorem:main} by moving to the setting of the mapping class group. However, we show here that this is not the case. 

\begin{lemma}\label{lemma:braidtomod}
Suppose that $s: B_n(S^2) \to B_{n,m}(S^2)$ is a section. Then there is a section $\bar{s}: \Mod_n(S^2) \to \Mod_{n,m}(S^2)$.
\end{lemma}
\begin{proof}
The relationship between the four groups under study is summarized by the following diagram, where the rows are short exact sequences:
\[
\xymatrix{
1 \ar[r] 	& \pair{\omega_{n+m}} \ar[r] \ar[d]	& B_{n,m}(S^2) \ar[r] \ar[d]_{p}	& \Mod_{n,m}(S^2) \ar[r] \ar[d]_{\bar p}	& 1\\
1 \ar[r]	& \pair{\omega_n} \ar[r] 		& B_n(S^2) \ar[r] \ar@/_/[u]_{s}	& \Mod_n(S^2) \ar[r]			& 1.
}
\]
Given $f \in \Mod_n(S^2)$, let $\tilde f \in B_n(S^2)$ be an arbitrary lift. We claim that 
\[
\bar s(f) := s(\tilde f) \pmod{\pair{\omega_{n+m}}}
\]
gives a well-defined section homomorphism. To see this, recall from Lemma \ref{lemma:2torsion} that $\omega_n \in B_n(S^2)$ is the unique element of order $2$ in $B_n(S^2)$, and that $\omega_{n+m}$ is similarly characterized as an element of $B_{n+m}(S^2)$. Since $s$ is a section, it follows that $s(\omega_n) = \omega_{n+m}$. Thus $\bar s(f)$ is well-defined as an element of $B_{n,m}(S^2)/\pair{\omega_{n+m}} \cong \Mod_{n,m}(S^2)$. It is straightforward to verify that $\bar s$ determines a homomorphism and that $\bar p \circ \bar s = \id$.
\end{proof}

\section{Cabling braids on the sphere}\label{section:cable} This section is dedicated to the proof of Theorem \ref{theorem:cabling}. We begin with a discussion of the {\em cabling construction}. We then analyze the obstruction to cabling in the setting of the spherical braid group, leading to the proof of Theorem \ref{theorem:cabling}. 

\subsection{Cabling} Cabling has a simple intuitive description. Given a braid $\beta$, one imagines each strand as actually being composed of a cable of smaller strands. This should furnish a homomorphism $c_k: B_n \to B_{nk}$, where $k$ is the number of strands in each cable. Making this precise requires additional data: one must specify the {\em internal} braid structure that each cable possesses. 

\begin{definition}[Cabling vector]\label{definition:cvec}
Let $\phi \in B_k$ be fixed, and let $a_1, \dots, a_{n-1}, c, t$ be arbitrary integers. The {\em cabling vector} is the tuple $v \in B_k \times \Z^{n+1}$ given by
\[
v := (\phi; a_1, \dots, a_{n-1}, c, t).
\]
\end{definition}

\begin{remark}\label{remark:decable}
As we are ultimately interested in constructing sections of $B_{n,m}(S^2) \to B_n(S^2)$, we will need to understand when it is possible to ``de-cable'', that is, when there exists a homomorphism $p: B_{nk} \to B_n$ such that $p \circ c_v = \id$. It is geometrically clear that such a $p$ exists whenever the element $\phi$ is contained in the subgroup $B_{k-1,1}$ consisting of braids where one strand is required to return to its starting point. In this case, $c_v$ is valued in the subgroup $B_{n, n(k-1)}$. 
\end{remark}

\begin{figure}
\labellist
\small
\endlabellist
\includegraphics{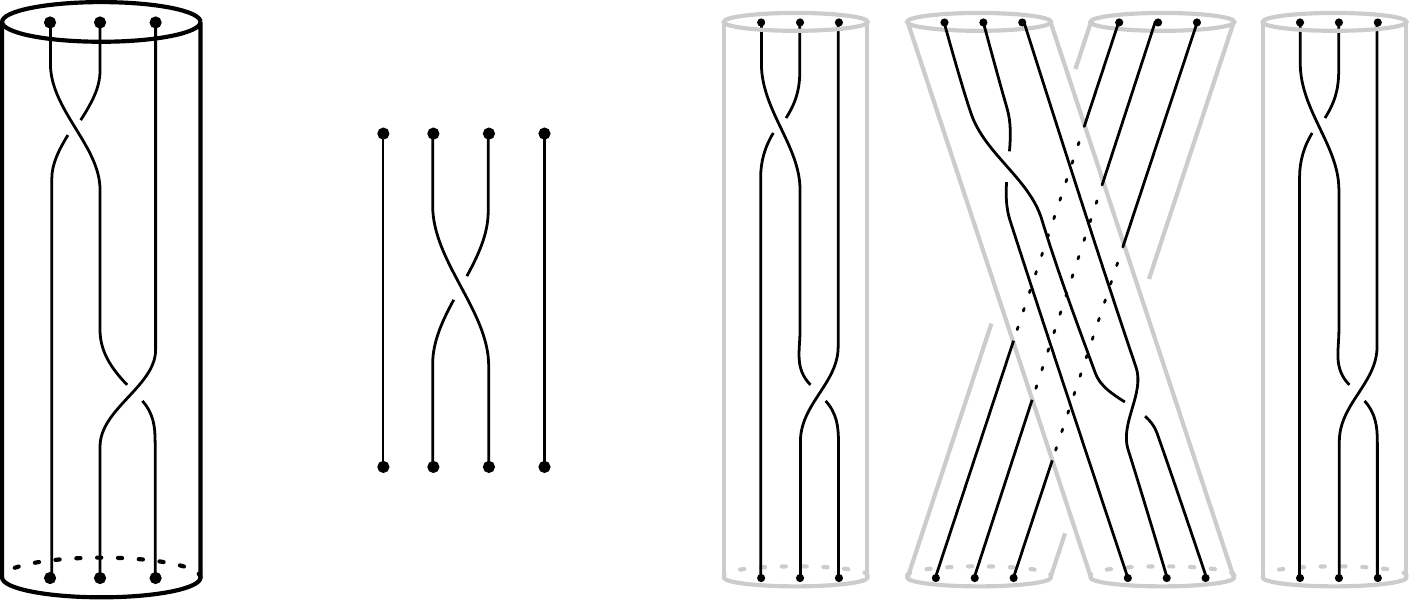}
\caption{The cabling procedure. Left to right: the braid $\phi \in B_3$, the generator $\sigma_2$, and the cabled braid $c_v(\sigma_2)$. Here $a_2 = 1, c = 1, t= 1$.}
\label{figure:cable}
\end{figure}

\begin{lemma}[Cabling construction]\label{lemma:cable}
Let $v = (\phi; a_1, \dots, a_{n-1}, c, t)$ be a cabling vector. Then there is a homomorphism
\[
c_v: B_n \to B_{nk}.
\]
Under $c_v$, each generator $\sigma_i$ is replaced by a cable: the $i^{th}$ strand is replaced with the braid $\phi^{a_i}$, the $(i+1)^{st}$ is replaced with $\phi^{t - a_i}$, and all remaining strands are replaced with the braid $\phi^c$. 
\end{lemma}
\begin{proof}
To prove the lemma, it suffices to show that the cabled generators $\{c_v(\sigma_i)\}$ satisfy all braid and commutation relations. These are both straightforward to check. For instance, consider the two braids $c_v(\sigma_i \sigma_{i+1} \sigma_i)$ and $c_v(\sigma_{i+1} \sigma_i \sigma_{i+1})$. In both braids, the $i^{th}$ strand is replaced by the braid $\phi^{a_i + a_{i+1} + c}$, the $(i+1)^{st}$ is replaced by $\phi^{t + c}$, the $(i+2)^{nd}$ is replaced by $\phi^{2t +c - a_i - a_{i+1}}$, and all other strands are replaced by $\phi^{3c}$. The commutation relation $c_v(\sigma_i \sigma_j) = c_v(\sigma_j \sigma_i)$ for $\abs{i-j} >1$ is similarly easy to verify. 
\end{proof}

\subsection{Cabling spherical braids} The cabling construction described above makes implicit use of a consistent framing of a neighborhood of each strand. Such framings are trivial to construct when the ambient space is the disk, but the usual topological constraints obstruct such framings for spherical braids. On the group-theoretic level, this obstruction can be formulated in terms of the following diagram, whose rows are the short exact sequences describing the spherical braid groups as quotients of the braid groups of the disk:
\begin{equation}\label{equation:diagram}
\xymatrix{
1 \ar[r]	& \pair{\pair{R_n}} \ar[r]\ar@{.>}[d]	& B_n \ar[r] \ar[d]_-{c_v}	& B_n(S^2) \ar[r] \ar@{.>}[d]_{\bar{c_v}} 	& 1\\
1 \ar[r]	& \pair{\pair{R_{nk}}} \ar[r]		& B_{nk} \ar[r]			& B_{nk}(S^2) \ar[r]						&1.
}
\end{equation}
The diagram shows that the homomorphism $c_v: B_n \to B_{nk}$ will descend to a homomorphism $\bar{c_v}: B_n(S^2) \to B_{nk}(S^2)$ if and only if the braid $R_n = \sigma_1 \dots \sigma_{n-1}^2 \dots \sigma_1$ is sent to an element of $\pair{\pair{R_{nk}}}$ (the normal closure of the element $R_{nk}$). In geometric terms, this is equivalent to the requirement that $c_v(R_n)$ be isotopic to the identity as a spherical braid. 

In order to understand the constraints this imposes on the cabling vector, we must understand the isotopy between the spherical braids $R_n$ and $1$ once framings are taken into account. The content of Lemma \ref{lemma:framing} below is that the isotopy $R_n \sim 1$ introduces a double twist to the framing.

\begin{figure}
\labellist
\small
\endlabellist
\includegraphics{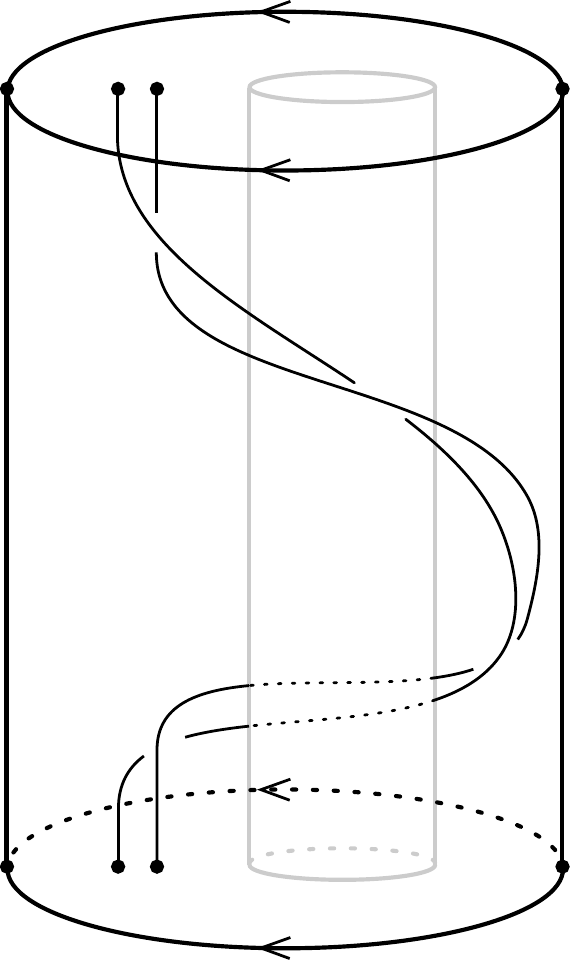}
\caption{The braid $R_n(k) \subset S^2 \times [0,1]$, depicted for $k = 2$. The remaining $(n-1)k$ strands are contained inside the grey cylinder and have constant $S^2$-coordinate. For general $k$, one can imagine the two depicted strands as determining the left and right edges of a flat strip on which the $k$ strands are arranged.}
\label{figure:framing}
\end{figure}

\begin{lemma}\label{lemma:framing}
For any $n \ge 3$ and $k \ge 2$, the braid $R_n(k)$ depicted in Figure \ref{figure:framing} is trivial as an element of $B_{nk}(S^2)$.
\end{lemma}
\begin{proof}
This can be directly verified by applying the isotopy shown in Figure \ref{figure:isotopy}. Alternatively, this can be understood in the setting of the mapping class group. Let $b$ denote the outer boundary component of the disk. Let $D_1$ be the disk around the first $k$ points, with boundary component $d$. Let $\gamma$ be the curve surrounding all points outside of $D_1$. As a mapping class, 
\[
\overline{R_n(k)}=DP(D_1, \gamma)T_{d}^{2},
\]
where $DP(D_1, \gamma)$ denotes the ``disk pushing map'' of the disk $D_1$ around curve $\gamma$ and $T_d$ denotes the Dehn twist about $d$. As in \cite[Page 119]{FM}, we can see that $DP(D_1, \gamma)=T_{b}T_{\gamma}^{-1}T_d^{-1}$. Therefore, $\overline{R_n(k)}=T_{b}T_{\gamma}^{-1}T_d^{-1}T_d^2=1\in \Mod_{nk}(S^2)$, since $T_b=1$ and $T_{\gamma}=T_d$ in $\Mod_{nk}(S^2)$.
\end{proof}

\begin{figure}
\labellist
\small
\endlabellist
\includegraphics[scale=0.8]{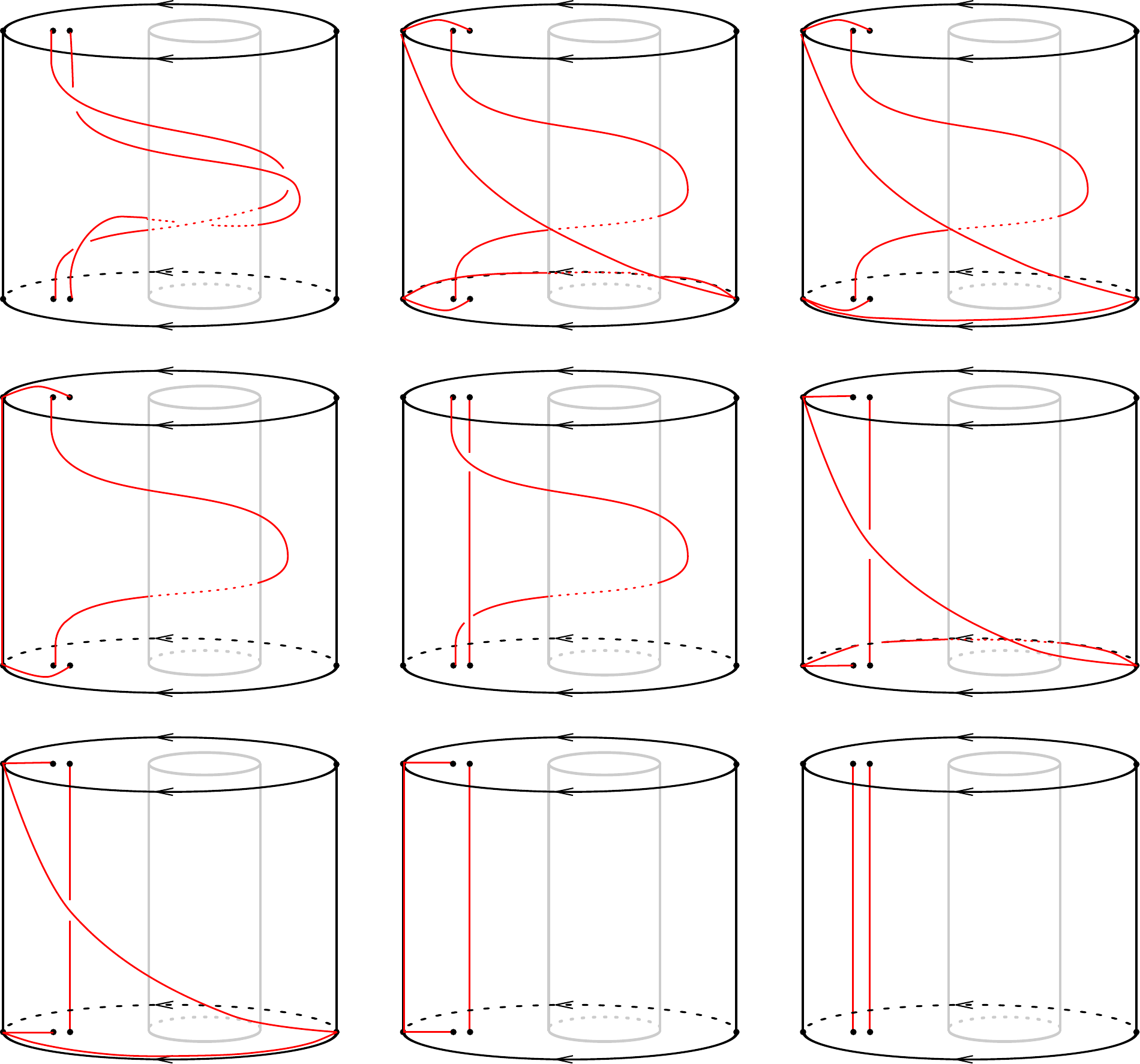}
\caption{The isotopy of Lemma \ref{lemma:framing}. The sequence should be read lexicographically. Between steps 2 and 3, a portion of the braid is pulled from the back to the front using the edge identification; the same move occurs between steps 6 and 7.}
\label{figure:isotopy}
\end{figure}

The braid $R_n(k)$ is a $k$-stranded cabling of $R_n$. In order for a given cabling homomorphism $c_v: B_n \to B_{nk}$ to descend to $\bar{c_v}: B_n(S^2) \to B_{nk}(S^2)$, it therefore suffices to produce a cabling vector $v = (\phi, a_1, \dots, a_{n-1}, c,t)$ for which $c_v(R_n) = R_n(k)$. Lemma \ref{lemma:existence} produces a class of suitable such $v$.

\begin{lemma}\label{lemma:existence}
Let $k = k'(n-1)(n-2) + 1$, and let $\phi \in B_{k-1,1}$ be the element $\phi = (\sigma_1 \dots \sigma_{k-2}\sigma_{k-1}^2)^{k'}$. For $a_1, \dots, a_{n-1}$ arbitrary, $c = -1$, and $t = 2n-4$, the twist vector $v = (\phi, a_1, \dots, a_{n-1}, c,t)$ satisfies 
\[
c_v(R_n) = R_n(k).
\]
\end{lemma}
\begin{proof}
We must analyze the internal braiding on each strand of the cabled braid 
\[
c_v(R_n) = c_v(\sigma_1 \dots \sigma_{n-1}^2 \dots \sigma_1).
\]
In order to determine $c_v(R_n)$, we track one strand at a time, applying the cabling construction (Lemma \ref{lemma:cable}) one letter at a time.

We begin with the first strand. Before applying the $i^{th}$ letter of the subword $\sigma_1 \dots \sigma_{n-1}$, the first strand is in position $i$, and so is cabled with $\phi^{a_i}$. Before applying the $i^{th}$ letter of the other subword $\sigma_{n-1} \dots \sigma_1$, the first strand is in position $n+1 - i$, and so is cabled with $\phi^{t-a_i}$. 

Altogether then, the first strand is cabled with $\phi^{(n-1)t} = \phi^{2(n-1)(n-2)}$. Geometrically, $\phi$ is represented as a rotation by an angle $2\pi/(n-1)(n-2)$, with one strand fixed at the center of the disk and the remaining strands arranged at equally-spaced points along a circle. Thus, $\phi^{2(n-1)(n-2)}$ is the braid given by two full twists of the strands about the central axis, which is exactly the cabling of the first strand in the braid $R_n(k)$. 

The cabling on the remaining strands can be determined in a similar fashion. Fix $j \ge 2$. The subword $\sigma_1 \dots \sigma_{j-2}$ leaves the $j^{th}$ strand in position $j$, and so gives a total cabling of $\phi^{2-j}$. Then $\sigma_{j-1}$ moves the $j^{th}$ strand to position $j-1$ and appends $\phi^{t-a_{j-1}}$ to the cabling. The subword $\sigma_{j} \dots \sigma_{n-1}^2\dots \sigma_{j}$ leaves the $j^{th}$ strand in position $j-1$, appending $\phi^{2j-2n}$. Then $\sigma_{j-1}$ moves the $j^{th}$ strand to position $j$, appending $\phi^{a_{j-1}}$ to the cabling. Finally the remaining subword $\sigma_{j-2} \dots \sigma_1$ appends an additional $\phi^{2-j}$. The total cabling induced by this procedure is $\phi^{t+4 -2n} = \id$, since $t = 2n-4$. Thus, $c_v(R_n) = R_n(k)$ as claimed.
\end{proof}

The last preliminary point to address before proving Theorem \ref{theorem:cabling} concerns the assertion mentioned in the Introduction that a section $s: B_n(S^2) \to B_{n,m}(S^2)$ can be promoted to a section $S: \Conf_n(S^2) \to \Conf_{n,m}(S^2)$. This is a standard argument in obstruction theory; see \cite[Proposition 4]{GG} for a written account.

\begin{proposition}[Gon\c calves-Guaschi]\label{proposition:promotion}
For $n \ge 3$, the fiber bundle $\Conf_{n,m}(S^2) \to \Conf_n(S^2)$ admits a section if and only if there is a group-theoretic section $s: B_n(S^2) \to B_{n,m}(S^2)$. 
\end{proposition}

Theorem \ref{theorem:cabling} now follows from the preceding analysis.

\begin{proof} {\em (of Theorem \ref{theorem:cabling})}
Suppose $m$ is divisible by $n(n-1)(n-2)$, and set $k = m/n + 1$. Let $v$ be a cabling vector satisfying the hypotheses of Lemma \ref{lemma:existence}. Applying Lemma \ref{lemma:existence} and appealing to (the analysis following) diagram \eqref{equation:diagram}, it follows that $c_v$ descends to a homomorphism 
\[
\bar{c_v}: B_n(S^2) \to B_{nk}(S^2).
\]
Note that $\phi = (\sigma_1 \dots \sigma_{k-2} \sigma_{k-1}^2)^{k'}$ is an element of $B_{k-1,1}$. By Remark \ref{remark:decable}, $\bar{c_v}$ is valued in the subgroup $B_{n, n(k-1)}(S^2) = B_{n,m}(S^2)$ and so provides a section of $p: B_{n,m}(S^2) \to B_n(S^2)$ as desired. By Proposition \ref{proposition:promotion}, $\bar{c_v}$ can be promoted to a section $C_v: \Conf_n(S^2) \to \Conf_{n,m}(S^2)$. 
\end{proof}

\begin{remark}\label{remark:spacelevel}
There is also an explicit construction of a section $S: \Conf_n(S^2) \to \Conf_{n,m}(S^2)$ that induces the cabling map $s: B_n(S^2) \to B_{n,m}(S^2)$. To construct this, given an ordered triple $(z_1, z_2, z_3)$ of distinct points in $\CP^1$, we let $M_{z_1, z_2, z_3}$ be the unique M\"obius transformation taking $(z_1, z_2, z_3) \to (0,1, \infty)$. This determines a map
\[
M: \Conf_{n-3,1,1,1}(\CP^1)\times \CP^1 \to \CP^1,
\]
where $\Conf_{n-3,1,1,1}(\CP^1)$ is the cover of $\Conf_n(\CP^1)$ consisting of configurations with three ordered distinguished points. In the fiber over $(\{z_4, \dots, z_n\}, z_1, z_2, z_3)$, the map $M$ is given by $M_{z_1,z_2,z_3}$. 

For a fixed $(\{z_4, \dots, z_n\}, z_1, z_2, z_3) \in \Conf_{n-3,1,1,1}(\CP^1)$, the product
\[
R_{z_3} := \prod_{z_i \ne z_j,\ i,j \ne 3} M_{z_1,z_2,z_3}
\]
determines a rational map of degree $(n-1)(n-2)$ with a pole at $z_3$ of order $(n-1)(n-2)$ and zeroes at all $z_j, j \ne 3$ of order $(n-2)$. It is clear from the construction that $R_{z_3}$ is in fact well-defined given only a point $(\{z_j, j \ne 3\}, z_3) \in \Conf_{n-1,1}(\CP^1)$. As before, this determines a map
\[
R: \Conf_{n-1,1}(\CP^1) \times \CP^1 \to \CP^1.
\]
Let $v_0$ be a fixed normal vector near $\infty\in \mathbb{C}P^2$. 

It is not hard to construct a continuous function $\epsilon: \Conf_{n}(\CP^1) \to \C$ with the property that the collection of $n(n-1)(n-2)$ points
\begin{equation}\label{equation:newpoints}
\bigcup_{i = 1}^n R_{z_i}^{-1}(\epsilon(\{z_1,...,z_n\}))
\end{equation}
is distinct and disjoint from the set $\{z_1, \dots, z_n\}$ ($\epsilon$ should be thought of as being very very large: then $R_{z_i}^{-1}(\epsilon(\{z_1,...,z_n\}))$ is a collection of $(n-1)(n-2)$ points very near $z_i$). Moreover the set \eqref{equation:newpoints} is well-defined given only the point $\{z_1, \dots, z_n\} \in \Conf_n(\CP^1)$. Taking $k$ such functions $\epsilon_1, \dots, \epsilon_k$ with pairwise-disjoint images (we can arrange them on different lines towards $\infty$), one can construct continuous sections $s: \Conf_n(\CP^1) \to \Conf_{n,m}(\CP^1)$ for any $m$ divisible by $n(n-1)(n-2)$. 
\end{remark}

\section{Three and four points}\label{section:34}
In this section, we give some algebro-geometric constructions of sections of the bundles $\Conf_{3,m}(\CP^1) \to \Conf_3(\CP^1)$ and $\Conf_{4,m}(\CP^1) \to \Conf_4(\CP^1)$. These results are summarized in Theorem \ref{theorem:34}. 

\subsection{Three points} The space $\Conf_3(S^2)$ is of course very special. Under the identification $S^2 = \CP^1$, the group $\Aut(\CP^1) = \PGL_2(\C)$ acts triply-transitively on $\CP^1$, i.e. transitively on $\Conf_3(\CP^1)$. Thus one method for constructing sections is to first apply an automorphism to normalize the configuration to the set $\{0,1,\infty\}$, and then find all possible configurations of points distinct from $\{0,1,\infty\}$ and invariant under the stabilizer of $\{0,1,\infty\}$ in $\Aut(\CP^1)$. We will flesh out this approach by means of the {\em cross-ratio}. Recall from Remark \ref{remark:spacelevel} the M\"obius transformation $M_{z_1,z_2,z_3}$ characterized by sending the triple $(z_1,z_2,z_3)$ to $(0,1,\infty)$.

\para{The cross-ratio} Let $z_1, z_2, z_3, z_4 \in \CP^1$ be four ordered points. The {\em cross-ratio} is the expression
\[
[z_1,z_2; z_3, z_4] := M_{z_1,z_2,z_3}(z_4) = \frac{(z_2-z_3)(z_4-z_1)}{(z_2-z_1)(z_4-z_3)}.
\]
To see how the cross-ratio can be exploited to construct sections of $\Conf_{3,n}(\CP^1) \to \Conf_3(\CP^1)$, it is necessary to understand how the value of $[z_1,z_2;z_3,z_4]$ changes under a permutation $\sigma \in S_4$. 

\begin{lemma}\label{lemma:perm}
Let $z_1, z_2,z_3,z_4 \in \CP^1$ be given, and suppose that $[z_1,z_2;z_3,z_4] = \lambda$. Let $\sigma \in S_4$ be an arbitrary permutation. Then
\[
[z_{\sigma(1)}, z_{\sigma(2)}; z_{\sigma(3)}, z_{\sigma(4)}] \in \left\{\lambda, \frac{1}{\lambda}, 1-\lambda, \frac{1}{1-\lambda}, \frac{\lambda - 1}{\lambda}, \frac{\lambda}{\lambda-1} \right\}.
\]
Thus the cross-ratio determines a generically 6-valued function $[\{z_1, z_2, z_3, z_4\}]$ of unordered $4$-tuples.
\end{lemma}

\begin{remark}\label{remark:stabilizer}
It is easy to see that the stabilizer of $\{0,1,\infty\}$ in $\PGL_2(\C)$ is the dihedral group $D_3$. The six values $\left\{\lambda, \frac{1}{\lambda}, 1-\lambda, \frac{1}{1-\lambda}, \frac{\lambda - 1}{\lambda}, \frac{\lambda}{\lambda-1} \right\}$ in fact comprise the orbit of $\lambda$ under $D_3$.
\end{remark}

Lemma \ref{lemma:perm} also allows us to view the cross-ratio as a multi-valued function on $\Conf_3(\CP^1)$. Given  $\{z_1, z_2, z_3\} \in \Conf_3(\CP^1)$ and $\lambda \in \CP^1$, define
\[
\times(z_1,z_2, z_3, \lambda) = \{z_4 \in \CP^1 \mid \lambda \in [\{z_1, z_2, z_3, z_4\}] \}.
\]
For generic values of $\lambda$, the function $\times$ is $6$-valued. However, $\times$ is $3$-valued for $\lambda \in \{-1,\frac{1}{2}, 2\}$, and is $2$-valued for $\lambda = \zeta^{\pm 1}$ either of the primitive sixth roots of unity. Moreover, 
\[
\times(z_1, z_2, z_3, \lambda) \cap\times(z_1, z_2, z_3, \lambda') = \emptyset
\]
whenever $\lambda$ and $\lambda'$ lie in different orbits of $D_3$, and 
\[
\times(z_1, z_2, z_3, \lambda) \cap \{z_1, z_2, z_3\} = \emptyset
\] as long as $\lambda \ne 0, 1, \infty$. 

\begin{proposition}\label{proposition:3pts}
For any $m \ge 0$ satisfying $m \equiv 0,2 \pmod 3$, there exists an algebraic section $\sigma$ of the bundle $\Conf_{3,m}(\CP^1) \to \Conf_3(\CP^1)$. Moreover, $\sigma$ is conformally invariant, i.e. equivariant with respect to the action of $\PGL_2(\C)$. 
\end{proposition}
\begin{proof}
There is a unique expression for $m$ of the form
\[
m = 2 a + 3 b + 6 c
\]
with $a, b \in \{0,1\}$. Set $k = a + b + c$. Choose a set $\{\lambda_1, \dots, \lambda_k\} \subset \CP^1 \setminus \{0,1,\infty\}$; these points should lie in distinct orbits under the action of the stabilizer of $\{0,1,\infty\}$. If $a = 1$ then set $\lambda_1 = \zeta$; likewise, if $b = 1$ then set $\lambda_2 = -1$. Then the assignment
\[
\sigma(\{z_1, z_2, z_3\}) = \bigcup_{i = 1}^k \times(z_1,z_2, z_3, \lambda_i)
\]
has the required properties. 
\end{proof}

\subsection{Four points}\label{section:4} We now turn to the problem of constructing sections of the bundle $\Conf_{4,m}(\CP^1) \to \Conf_4(\CP^1)$. We are grateful to Ian Frankel for the suggestion to look at torsion points on elliptic curves. The basic fact underlying the constructions in this section is the following well-known result.

\begin{lemma}\label{lemma:elliptic}
Let $S = \{z_1, z_2, z_3, z_4\} \subset \CP^1$ be an arbitrary $4$-tuple of distinct points. Then there exists an elliptic curve $(E_S,*)$ (with identity element $* \in E_S$) such that under the elliptic involution $\iota: E_S \to \CP^1$, the branch locus in $\CP^1$ is the set $S$. The preimage $\iota^{-1}(S)$ is the set of $2$-torsion points of $(E_S, *)$. 
\end{lemma}

Lemma \ref{lemma:elliptic} leads to the construction of sections of $\Conf_{4,m}(\CP^1) \to \Conf_4(\CP^1)$ for many values of $m$. To formulate the result, let $P(k)$ denote the number of primitive elements of the group $(\Z/k\Z)^2$. An explicit formula for $P(k)$ can be obtained from the observation that $P(p^k) = p^{2k}-p^{2k-2}$ for any prime $p$, in combination with the fact that $P$ is evidently a multiplicative function.

\begin{proposition}\label{proposition:4pts}
Let $m$ be a positive integer of the form $m = 2k^2 - 2$ or $m = \frac{P(4k)}{2}$. Then there exists an algebraic section $\sigma$ of the bundle $\Conf_{4,m}(\CP^1) \to \Conf_4(\CP^1)$. Moreover, $\sigma$ is conformally invariant.
\end{proposition}

\begin{proof}
First consider the case $m = 2k^2-2$. Let $S = \{z_1, z_2, z_3, z_4\} \in \Conf_4(\CP^1)$ be given, and let $(E_S, *)$ be the elliptic curve of Lemma \ref{lemma:elliptic}. The $2k$-torsion subgroup of $(E_S,*)$ has cardinality $4k^2$. Since $2k$ is even, it follows that this set does not depend on which of the four points $\iota^{-1}(z_i)$ is chosen as the identity element. Among these points, exactly four are the $2$-torsion points $\iota^{-1}(S)$. The elliptic involution $x \mapsto -x$ restricts to a {\em free} involution on the remaining $4k^2 - 4$ points. Under $\iota$, these points descend to a set of $2k^2 - 2$ distinct points on $S^2$ that are necessarily disjoint from $S$. The continuity and conformality of this construction are clear. 

The construction for $m = \frac{P(4k)}{2}$ proceeds along similar lines. The set of {\em primitive} $4k$-torsion points is well-defined independently of the choice of origin among the points $\iota^{-1}(S)$, and has cardinality $P(4k)$ by definition. As before, this descends under the elliptic involution to a set of cardinality $\frac{P(4k)}{2}$ in $S^2$. 
\end{proof}

The first few such values of $m$ are given by $m = 6,16, 24, 30, 48, 70$, corresponding respectively to the $4$-torsion, $6$-torsion, primitive $8$-torsion, $8$-torsion (as well as primitive $12$-torsion), and $12$-torsion points. Not all values of $m$ that are ``unobstructed'' in the sense of Theorem \ref{theorem:GG} are represented, although all four allowable residues do appear. For instance, $m = 22$ is unobstructed and yet does not appear on the above list. However, the results are sufficient to prove Theorem \ref{theorem:34}.

\begin{proof}{\em (of Theorem \ref{theorem:34})} The assertions concerning the case $n = 3$ are subsumed by Proposition \ref{proposition:3pts}. The assertions concerning $n=4$ follow readily from Proposition \ref{proposition:4pts} and Theorem \ref{theorem:cabling}. If $m \ge 70$ is congruent to one of the four allowable values $0,6,16,22$ mod $24$, then one can produce a section $\Conf_4(S^2) \to \Conf_{4,m}(S^2)$ by combining the construction of Proposition \ref{proposition:4pts} (for $1$-torsion, $4$-torsion, $6$-torsion, $12$-torsion respectively for $0,6,16,22$ mod $24$) with the cabling construction of Theorem \ref{theorem:cabling}.
\end{proof}

\section{Canonical reduction systems}\label{section:CRS}
The goal of this section is to outline the portion of the theory of canonical reduction systems needed for the proof of Theorem \ref{theorem:main}. We first recall the Nielsen-Thurston classification of elements of $\Mod(S)$, where $S$ is an arbitrary surface of finite type. For this discussion, and for the remainder of the paper, we invoke the usual conventions concerning isotopy: by ``curve'', we really mean ``isotopy class of curve'', by ``disjoint'' we really mean ``existence of disjoint isotopy class representatives'', etc. 

With these stipulations in place, the Nielsen-Thurston classification asserts that each $f \in \Mod(S)$ is exactly one of the following types: {\em periodic, reducible,} or {\em pseudo-Anosov}. A mapping class $f$ is periodic if $f^n = \id$ for some $n \ge 1$, and is reducible if there is some essential multicurve $\gamma \subset S$ fixed (as a set, not necessarily component-wise) by $f$. Otherwise, $f$ is said to be pseudo-Anosov. 

\begin{definition}[Canonical reduction system]
Let $f \in \Mod(S)$ be given. A {\em reduction system} for $f$ is any essential multicurve $\gamma = \{c_1, \dots, c_n\}$ fixed setwise by $f$. A reduction system is {\em maximal} if it is maximal with respect to inclusion of reduction systems for $f$. The {\em canonical reduction system} for $f$, written $\CRS(f)$, is defined to be the intersection of all maximal reduction systems for $f$.
\end{definition}

Canonical reduction systems provide a sort of Jordan form for mapping classes. The role of Jordan blocks is played by the components of the cut-open surface 
\[
S_{\CRS(f)} := S \setminus \CRS(f).
\]
The lemma below follows from \cite[Theorem C]{BLM}; see also \cite[Corollary 13.3]{FM}.
\begin{lemma}\label{lemma:restriction}
Let $f \in \Mod(S)$ be given, and suppose that $f$ preserves some component $S_i$ of $S_{\CRS(f)}$ and so induces an element $f_i \in \Mod(S_i)$. Then $f_i$ is either periodic or else pseudo-Anosov.
\end{lemma}

\begin{remark}\label{remark:torsion} In this paper, we are exclusively interested in the case where $S$ is a punctured sphere. Then each component $S_i$ is also a punctured sphere, and so the classification of torsion elements of $\Mod(S_i)$ given in the table \eqref{table:torsion} is applicable. In particular, we see that if $f_i \in \Mod(S_i)$ is periodic and fixes at least three punctures, then $f_i$ is trivial, and any remaining punctures in $S_i$ must also be fixed. We note that by our definitions, a boundary component of $S_i$ (when viewed as a subsurface of $S$) is treated as a puncture when $S_i$ is viewed as an abstract punctured sphere.
\end{remark}

Canonical reduction systems behave as expected under conjugation. We record the following lemma for later use; its proof is trivial.

\begin{lemma}\label{lemma:crsconj}
Let $f,g \in \Mod(S)$ be given. Then
\[
\CRS(f g f^{-1}) = f(\CRS(g)).
\]
In particular, if $f$ and $g$ commute, then $f(\CRS(g)) = \CRS(g)$. 
\end{lemma}

If mapping classes $f, g$ commute, then $\CRS(f)$ and $\CRS(g)$ satisfy an especially nice relationship; see \cite[Proposition 2.6]{lei3}.
\begin{lemma}\label{lemma:crscommute}
Suppose that $f, g \in \Mod(S)$ commute. Then each component of $\CRS(g)$ is either also a component of $\CRS(f)$, or else is disjoint from each component of $\CRS(f)$. 
\end{lemma}

We conclude this section with a useful lemma giving a criterion for the equality of two subsurfaces of a punctured sphere.

\begin{lemma}\label{lemma:partition}
Let $S$ and $S'$ be two subsurfaces of a punctured sphere $\Sigma$. Suppose that the boundaries $\partial S$ and $\partial S'$ have the same number of components, and that each component of $\partial S$ is either disjoint from each component of $\partial S'$, or else is also a component of $\partial S'$. Suppose further that no component of $\partial S'$ is contained in the interior of $S$. If $S$ and $S'$ contain the same number of punctures and there is a puncture $x$ contained in both $S$ and $S'$, then in fact $S$ and $S'$ determine the same isotopy class of subsurface. 
\end{lemma}
\begin{proof}
The Euler characteristic of either surface is determined by the number of boundary components and the number of punctures contained in the interior. As each surface is a punctured sphere, it follows that moreover, the homeomorphism type is determined by this data, and hence the assumptions imply that $S$ and $S'$ are abstractly homeomorphic. Since no component of $\partial S'$ is contained in the interior of $S$, and since $S$ and $S'$ contain some common puncture, it follows that there is a containment of subsurfaces $S \subset S'$. Since $S$ and $S'$ are assumed to contain the same number of punctures, these must each be contained in $S$. It follows that each boundary component of $S'$ is isotopic to a boundary component of $S$, and the result follows.
\end{proof}

\section{Proof of Theorem \ref{theorem:main}: Preliminaries}\label{section:prelims}
This is the first of five sections dedicated to the proof of Theorem \ref{theorem:main}. The plan is as follows. In Section \ref{section:prelims}, we establish some preliminary ideas. This allows us to give a high-level overview of the proof in Section \ref{section:overview} and to divide the ensuing argument up into two cases A and B. In Section \ref{section:reducible} we prove a pair of crucial lemmas. The arguments for cases A and B are carried out in Sections \ref{section:A} and \ref{section:B}, respectively.

Throughout the proof, fix $n \ge 6$. We remind the reader of the terminology of ``old points'' $\{x_1, \dots, x_n\}$ and ``new points'' $\{y_1, \dots, y_m\}$ of Definition \ref{definition:oldnew}. For the sake of contradiction, we assume that $m$ is the least integer not divisible by $n(n-1)(n-2)$ for which a section $s: B_n(S^2) \to B_{n,m}(S^2)$ exists. By Lemma \ref{lemma:braidtomod}, a section $s: B_n(S^2) \to B_{n,m}(S^2)$ induces a section $s: \Mod_n(S^2) \to \Mod_{n,m}(S^2)$. For the remainder of the proof, we will work in the setting of the mapping class group. We define
\[
\Gamma := s(\Mod_n(S^2)) \leqslant \Mod_{n,m}(S^2).
\]
Before we can give the overview of the proof in the next section, there are three preliminary results that need to be established. In Section \ref{subsection:trans}, we show that $\Gamma$ acts transitively on the set of new points (Lemma \ref{lemma:trans}). In Section \ref{subsection:torsion}, we show that some torsion element fixes a new point (Lemma \ref{lemma:torsion}). Finally in Section \ref{subsection:tree}, we study the canonical reduction system $\CRS(s(\sigma_1))$ and attach to this a tree in a canonical way (Lemma \ref{lemma:tree}).

\subsection{Transitivity on new points}\label{subsection:trans} A first observation to be made is that our hypotheses on $m$ imply that the action of $\Gamma$ on the set of new points is transitive.

\begin{lemma}\label{lemma:trans}
Let $m$ be the minimal integer not divisible by $n(n-1)(n-2)$ for which a section $s: \Mod_n(S^2) \to \Mod_{n,m}(S^2)$ exists. Then $\Gamma$ acts transitively on the set of new points.
\end{lemma}
\begin{proof}
If $\Gamma$ does not act transitively on the set of new points, then there exists some nontrivial $\Gamma$-invariant partition of $\{y_1, \dots, y_m\}$. Let $m'$ denote the cardinality of some part; by forgetting all points not in this part, there is a section $s': \Mod_n(S^2) \to \Mod_{n,m'}(S^2)$. As $m$ is not divisible by $n(n-1)(n-2)$, any nontrivial partition of an $m$-element set necessarily has some part of cardinality $m'<m$ not divisible by $n(n-1)(n-2)$. Such $m'$ contradicts the minimality of $m$. 
\end{proof}

\subsection{Fixed points of torsion elements}\label{subsection:torsion} The essential distinction between the case $m \equiv 0 \pmod {n(n-1)(n-2)}$, where sections of $\Conf_{n,m}(S^2) \to \Conf_n(S^2)$ exist, and $m \not \equiv 0 \pmod{n(n-1)(n-2)}$, where they do not, turns out to be the fact, recorded in Lemma \ref{lemma:torsion} below, that in the latter cases, there always exists some torsion element $\alpha$ that fixes at least one new point. 

\begin{lemma}\label{lemma:torsion}
Suppose that $m \not \equiv 0 \pmod {n(n-1)(n-2)}$, and that a section $s: \Mod_n(S^2) \to \Mod_{n,m}(S^2)$ exists. Then at least one of $\alpha \in \{\alpha_0, \alpha_1\} \subset \Mod_{n}(S^2)$ has the property that $s(\alpha)$ fixes some new point $A$.
\end{lemma}
\begin{proof}
We first claim that if a section exists, necessarily $m \equiv 0 \pmod {n-2}$. To see this, we study $\alpha_2 \in \Mod_n(S^2)$. This fixes two old points, and hence $s(\alpha_2)$ also fixes these points. By Remark \ref{remark:torsion}, $s(\alpha_2)$ has no further fixed points. Thus the set of $m$ new points decomposes as a union of $s(\alpha_2)$-orbits, each of cardinality $n-2$. 

It follows that if $m$ is not divisible by $n(n-1)(n-2)$, then $m$ is not divisible by at least one of $n$ or $n-1$. If $m \equiv k \pmod n$ for some integer $1 \le k < n$, then the action of $s(\alpha_0)$ on the set of new points has $k > 0$ fixed points. Similar reasoning shows that $s(\alpha_1)$ has a new fixed point whenever $m \not \equiv 0 \pmod{n-1}$. 
\end{proof}

\subsection{Canonical reduction systems and trees}\label{subsection:tree} We come now to the key object of interest. We will study the set
\[
\mathscr C:= \CRS(s(\sigma_1)). 
\]
The structure of $\mathscr C$ is best encoded as a graph.

\begin{definition}
The graph $\mathscr T$ has vertices in bijection with the components of $S^2_{\mathscr C}$, and edges in bijection with elements of $\mathscr C$. An edge $c \in \mathscr C$ joins the components $S_1, S_2 \subset S^2_{\mathscr C}$ for which $c$ is a boundary component of both $S_1$ and $S_2$.
\end{definition}

\begin{lemma}\label{lemma:tree}
The graph $\mathscr T$ is a tree.
\end{lemma}
\begin{proof}
It is clear from the construction that $\mathscr T$ is connected. Let $V, E$ denote the number of vertices and edges of $\mathscr T$, respectively. As $\mathscr T$ is connected, it follows that $\mathscr T$ is a tree if and only if the Euler characteristic satisfies
\[
\chi(\mathscr T) = V-E = 1.
\]
Enumerate the components of $S^2_{\mathscr C}$ as $S_1, \dots, S_V$. A component $S_i$ of $S^2_{\mathscr C}$ has Euler characteristic $2 - b_i$, where $b_i$ is the number of boundary components of $S_i$, i.e. the number of edges of $\mathscr T$ incident to $S_i$. Since each pair $S_i, S_j$ of components of $S^2_{\mathscr C}$ meet in $S^2$ along a union of circles (each of Euler characteristic zero), the cut-and-paste formula
\[
\chi(A \cup B) = \chi(A) + \chi(B) - \chi(A \cap B)
\] 
for the Euler characteristic gives the following expression for $\chi(S^2)$:
\[
2 = \chi(S^2) = \sum_{i = 1}^V (2-b_i) = 2V -  \sum_{i = 1}^V b_i = 2V -2 E.
\]
The result follows.
\end{proof}

\section{Proof of Theorem \ref{theorem:main}: Overview}\label{section:overview}
As we have already remarked, our standing assumption is that $n \ge 6$ and that $m$ is the minimal integer not divisible by $n(n-1)(n-2)$ for which a section $s: \Mod_n(S^2) \to \Mod_{n,m}(S^2)$ exists; Lemma \ref{lemma:trans} implies that $\Gamma$ acts transitively on the set of new points. Our strategy will be to derive a contradiction to the transitivity assumption, or else to show that $\Gamma$ is {\em reducible}: there exists a nonempty set $\mathscr R$ of disjoint essential curves satisfying $\Gamma(\mathscr R) = \mathscr R$. By Lemmas \ref{lemma:reduciblesame} and \ref{lemma:reducibledistinct} below, this will also produce a contradiction. Lemmas \ref{lemma:reduciblesame} and \ref{lemma:reducibledistinct} are established in Section \ref{section:reducible}.

\begin{lemma}\label{lemma:reduciblesame}
Fix $n\ge 3$, and let $s: \Mod_n(S^2) \to \Mod_{n,m}(S^2)$ be a section of $p: \Mod_{n,m}(S^2) \to \Mod_n(S^2)$. Suppose that $\Gamma$ acts transitively on the set of new points, and that there is a $\Gamma$-invariant subsurface $S \subset S^2$ that contains at least one old point. Then either $m$ is divisible by $n(n-1)(n-2)$, or else there is some $m' < m$ with $m'$ not divisible by $n(n-1)(n-2)$ and a section $s': \Mod_n(S^2) \to \Mod_{n,m'}(S^2)$.
\end{lemma}

\begin{lemma}\label{lemma:reducibledistinct}
Fix $n\ge 3$, and let $s: \Mod_n(S^2) \to \Mod_{n,m}(S^2)$ be a section of $p: \Mod_{n,m}(S^2) \to \Mod_n(S^2)$. Suppose that $\Gamma$ acts transitively on the set of new points, and that there is a $\Gamma$-invariant set $\{S_1, \dots, S_n\}$ of subsurfaces, each with a single boundary component $c_i$, such that $x_i \in S_i$ for $i = 1, \dots, n$. Then $m$ is divisible by $n(n-1)(n-2)$.
\end{lemma}

The argument proceeds by studying some distinguished components of $S^2_{\mathscr C}$. For $i=3,...,n$, let $S_i$ be the component of $S^2_{\mathscr C}$ that contains the old point $x_i$. The $S_i$ are not necessarily pairwise distinct. To get a better understanding of the set $\{S_i\}$, we make the following observations. By Lemma \ref{lemma:crsconj}, if $g\in \text{Mod}_{n}(S^2)$ commutes with $\sigma_1$, then $s(g)$ induces a permutation of $\mathscr C$ and hence an automorphism $g_*$ of $\mathscr T$. Moreover, the following lemma shows that there is a large supply of such elements $g$ for which the behavior on the set of old points is prescribed. The proof is elementary and is omitted. 
\begin{lemma}\label{equidistance}
For any pair of distinct old points $x_i, x_j$ with $i, j \ge 3$, there exists an element $g\in \Mod_n(S^2)$ such that $g$ commutes with $\sigma_1$ and such that $g(x_i)=x_3$ and $g(x_j)=x_4$.
\end{lemma}

For any $g$ as in Lemma \ref{equidistance}, the induced automorphism $g_*$ of $\mathscr T$ is in fact an {\em isometry}, equipping $\mathscr T$ with the metric $\delta$ in which edges have length $1$. Since $s(g)$ permutes the components of $S^2_{\mathscr C}$ and $g(x_i)=x_3$, it follows that $g_*(S_i)=S_3$. Similarly $g_*(S_j)=S_4$. Thus in the metric graph $(\mathscr T, \delta)$,  
\[
\delta(S_i,S_j)=\delta(g_*(S_i),g_*(S_j))=\delta(S_3,S_4).
\]
Therefore $\delta(S_i,S_j)=d$ a constant which does not depend on $i,j$. There are two possibilities: either (A) $d=0$, so that $S:= S_3 = \dots = S_n$, or else (B) $d>0$, so that each $S_3, \dots, S_n$ is a distinct subsurface of $S^2_{\mathscr C}$.

To analyze Case A, we appeal to the theory of canonical reduction systems. Since $s(\sigma_1)$ fixes a point $x_3\in S$, it follows that $s(\sigma_1)$ fixes the component $S$. Lemma \ref{lemma:restriction} then implies that the restriction of $s(\sigma_1)$ to $S$ is either pseudo-Anosov or else periodic. We handle each possibility in turn, as Cases A.1 and A.2, respectively. Case A.1 is resolved by showing that $S$ is necessarily $\Gamma$-invariant; this contradicts Lemma \ref{lemma:reduciblesame}. 

The analysis of Case A.2, where $s(\sigma_1)$ is assumed to be periodic, requires a further division into subcases. Lemma \ref{lemma:torsion} guarantees the existence of torsion elements of $\Gamma$ that fix at least one new point $A$. Case A.2 subdivides into Cases A.2.a and A.2.b, depending on whether $A$ is contained in $S$ or not. In Case A.2.a, where $A \in S$, we will show that either $A$ is a global fixed point, contradicting transitivity, or else that there is a nontrivial torsion element with $3$ fixed points, contradicting Remark \ref{remark:torsion}. In the alternative Case A.2.b, we will produce an essential $\Gamma$-invariant curve, contradicting Lemma \ref{lemma:reduciblesame}.

The ultimate aim in Case B is to show that $\Gamma$ is reducible. In Lemma \ref{lemma:c3}, we produce a collection $c_3, \dots, c_n$ of distinguished boundary components of $S_3, \dots, S_n$. After analyzing how $s(\alpha_0)$ acts on this set in Lemma \ref{lemma:sconj}, we are able to define two further curves $c_1, c_2$. We then show in Lemma \ref{lemma:invariant} that the set of curves $\{c_1, \dots, c_n\}$ is $\Gamma$-invariant, leading to a contradiction with Lemma \ref{lemma:reducibledistinct}.

\section{Proof of Theorem \ref{theorem:main}: The reducible case}\label{section:reducible}
In this section we treat the situation where $\Gamma := s(\Mod_n(S^2))$ is reducible. The objective is to prove Lemmas \ref{lemma:reduciblesame} and \ref{lemma:reducibledistinct}, reproduced for the reader's convenience below.\\

\noindent \textbf{Lemma \ref{lemma:reduciblesame}.} \textit{Fix $n\ge 3$, and let $s: \Mod_n(S^2) \to \Mod_{n,m}(S^2)$ be a section of $p: \Mod_{n,m}(S^2) \to \Mod_n(S^2)$. Suppose that $\Gamma$ acts transitively on the set of new points, and that there is a $\Gamma$-invariant subsurface $S \subset S^2$ that contains at least one old point. Then either $m$ is divisible by $n(n-1)(n-2)$, or else there is some $m' < m$ with $m'$ not divisible by $n(n-1)(n-2)$ and a section $s': \Mod_n(S^2) \to \Mod_{n,m'}(S^2)$.}

\begin{proof}\textit{(of Lemma \ref{lemma:reduciblesame})} A first observation is that $S$ contains {\em all} old points. Indeed, if $x_i \in S$, then for any $j = 1, \dots, n$, there exists $\phi_j \in \Gamma$ for which $\phi_j(x_i) = x_j$. As $x_i \in S$ and $S$ is $\Gamma$-invariant, it follows that $x_j \in S$ as well.

By hypothesis, $s$ is valued in the subgroup $\Mod_{n,m}(S^2, S)$ of mapping classes that preserve the subsurface $S$. There is a restriction map 
\[
r: \Mod_{n,m}(S^2,S) \to \Mod_n(S) \cong \Mod_{n,m'+m''}(S^2)
\]
where $m'$ is the number of boundary components of $S$, and $m''$ is the number of new points contained in $S$. Setting $s' := r \circ s$, 
we obtain a new homomorphism 
\[
s': \Mod_n(S^2) \to \Mod_{n,m'+m''}(S^2).
\]

We claim that $m'' = 0$, that $m' < m$, and that if $m'$ is divisible by $n(n-1)(n-2)$, then $m$ is as well. To see these claims, observe that since $S$ contains all of the old points, each component of $S^2 \setminus S$ contains only new points. Since each boundary component of $S$ is essential, there must be at least two new points contained in each component of $S^2 \setminus S$; this shows $m' < m$. By hypothesis, $\Gamma$ acts transitively on the set of new points. Since $S$ is $\Gamma$-invariant, any new points contained in $S$ cannot be exchanged with new points off of $S$, and so $m'' = 0$ as claimed. Moreover, $\Gamma$ must act transitively on the set of components of $S^2 \setminus S$. Letting $p$ denote the number of new points contained in each component, we see that $m = m' p$. Thus if $m'$ is divisible by $n(n-1)(n-2)$, so is $m$.

To establish Lemma \ref{lemma:reduciblesame}, it now suffices to show that $s'$ is a section of the forgetful map $p': \Mod_{n,m'}(S^2) \to \Mod_n(S^2)$. Recall that $s: \Mod_n(S^2) \to \Mod_{n,m}(S^2)$ is a section of the forgetful map $p: \Mod_{n,m}(S^2) \to \Mod_n(S^2)$. The claim now follows from the factorizations
\[
p|_{\Mod_{n,m}(S^2,S)} = p' \circ r
\]
and
\[
s' = r \circ s.
\]
\end{proof}

\noindent \textbf{Lemma \ref{lemma:reducibledistinct}.} \textit{Fix $n\ge 3$, and let $s: \Mod_n(S^2) \to \Mod_{n,m}(S^2)$ be a section of $p: \Mod_{n,m}(S^2) \to \Mod_n(S^2)$. Suppose that $\Gamma$ acts transitively on the set of new points, and that there is a $\Gamma$-invariant set $\{S_1, \dots, S_n\}$ of subsurfaces, each with a single boundary component $c_i$, such that $x_i \in S_i$ for $i = 1, \dots, n$. Then $m$ is divisible by $n(n-1)(n-2)$.}

\begin{proof}\textit{(of Lemma \ref{lemma:reducibledistinct})}
We claim that {\em all} new points are contained inside the set 
\[
\bigcup_{i=1}^n S_i.
\]
Certainly there must exist {\em some} new point in each $S_i$, as otherwise $c_i$ would be inessential. Since $\Gamma$ acts transitively on the set of new points and the set $\cup S_i$ is $\Gamma$-invariant, the claim follows.

To conclude the argument, we count the number of new points. As $\Gamma$ permutes the subsurfaces $S_i$, each contains the same number $m'$ of new points. For any $i = 1, \dots, n$, there is a conjugate $\alpha_{2,i}$ of $\alpha_2$ that fixes the point $x_i$. It follows that $S_i$ is $s(\alpha_{2,i})$-invariant, and hence the set of new points contained in $S_i$ decomposes as a union of orbits of $s(\alpha_{2,i})$. By Remark \ref{remark:powers}, each orbit contains $n-2$ points, so that $(n-2) \mid m'$. 

Likewise, let $\alpha_{1,i}$ be a conjugate of $\alpha_1$ that fixes $x_i$. Then $s(\alpha_{1,i})$ also fixes $S_i$ and so decomposes the new points in $S_i$ into a union of orbits. By Remark \ref{remark:powers}, each orbit contains $n-1$ points, so that also $(n-1) \mid m'$. We conclude that $(n-1)(n-2) \mid m'$, and as $m = n m'$, Lemma \ref{lemma:reducibledistinct} follows.
\end{proof}

\section{Proof of Theorem \ref{theorem:main}: Case A}\label{section:A}
The assumption in Case A is that $S:=S_3=S_4=...=S_n$. As discussed in the overview given in Section \ref{section:overview}, Case A divides into two subcases.

\noindent \underline{(1) $s(\sigma_1)$ is pseudo-Anosov on $S$. }  By Lemma \ref{lemma:alpha}, $\alpha_0^{2} \sigma_1\alpha_0^{-2}=\sigma_3$. Therefore 
\[
\CRS(s(\sigma_3)) = \CRS(s(\alpha_0^{2}\sigma_1\alpha_0^{-2})) = s(\alpha_0^2) \cdot \CRS(s(\sigma_1)) = s(\alpha_0^2) \cdot \mathscr C.
\]
Since $\sigma_1$ and $\sigma_3$ commute, Lemma \ref{lemma:crscommute} implies that any pair of curves $c \in \CRS(s(\sigma_3))$ and $d \in \CRS(s(\sigma_1)) = \mathscr C$ are disjoint or else equal. Let $\partial S \subset \mathscr C$ denote the set of boundary components of $S$. Then every element of $s(\alpha_0^2)(\partial S)$ is disjoint from the elements of $\partial S$, or else is also an element of $\partial S$. Since $s(\sigma_1)$ is pseudo-Anosov on $S$, Lemma \ref{lemma:crsconj} implies that none of the elements of $s(\alpha_0^2)(\partial S)$ are contained in the interior of $S$.

On the other hand, $S$ and $s(\alpha_0^2)(S)$ contain the same number of punctures and each contains the puncture $x_5$. Lemma \ref{lemma:partition} then implies that $\alpha_0^2(S)=S$. By the same reasoning, $\alpha_0^3(S)=S$, and it follows that $\alpha_0(S)=S$. By Lemma \ref{lemma:genset}, the entire group $\Gamma$ preserves the component $S$. Thus the hypotheses of Lemma \ref{lemma:reduciblesame} are satisfied, leading to a contradiction with the minimality assumption on $m$.

\noindent \underline{(2)  $s(\sigma_1)$ is periodic on $S$. } We first claim that in fact $s(\sigma_1)$ is the identity on $S$. This follows from the fact that $\sigma_1$ fixes at least the three old points $x_3, x_4, x_5 \in S$, in combination with Remark \ref{remark:torsion}. 

As in Lemma \ref{lemma:torsion}, let $\alpha$ be whichever of $\alpha_0, \alpha_1$ fixes some new point $A$. There are two possibilities: 

\begin{enumerate}[(a)]
\item $A\in S$:

Since $s(\sigma_1)$ acts by the identity on $S$, necessarily $s(\sigma_1)(A) = A$. Since also $s(\alpha)(A) = A$ by construction, it follows that the subgroup $G \leqslant \Mod_{n,m}(S^2)$ generated by $s(\sigma_1)$ and $s(\alpha)$ fixes $A$. If $\alpha = \alpha_0$, Lemma \ref{lemma:genset} implies that $G = \Gamma$. But then $\Gamma$ does not act transitively on the set of new points, in contradiction with Lemma \ref{lemma:trans}.

If $\alpha = \alpha_1$, Lemma \ref{lemma:genset2} implies that $G = s(\Mod_{n-1,1}(S^2))$. Thus $s$ restricts to give an injective homomorphism
\[
s: \Mod_{n-1,1}(S^2) \to \Mod_{n-1,1,m-1,1}(S^2).
\]
The element $\alpha_2 \in \Mod_{n}(S^2)$ is contained in $\Mod_{n-1,1}(S^2)$ and is torsion of order $n-2$ with two fixed points $x_{n-1}, x_n$, both old. Thus $s(\alpha_2)$ must also be torsion of order $n-2$ with two fixed old points. By Remark \ref{remark:torsion}, $s(\alpha_2)$ cannot have any further fixed points, but by definition every element of $\Mod_{n-1,1,m-1,1}(S^2)$ fixes the new point $A$, a contradiction.

\item $A\notin S$: 

In this case, there exists a curve in $c\in \partial S$ separating $x_3,...,x_n$ from $A$.  Such a $c$ is necessarily $s(\sigma_1)$-invariant, since $s(\sigma_1)$ acts as the identity on $S$. We claim that $c$ must also be $s(\alpha)$-invariant, and must moreover preserve the subsurfaces on either side of $c$. 
\begin{proof}
Since $n \ge 6$, Lemma \ref{lemma:alpha} implies that $\alpha^2\sigma_1\alpha^{-2}=\sigma_3$, and thus $s(\alpha^2)(c)$ belongs to $\CRS(s(\sigma_3))$. It follows that $s(\alpha^2)(c)$ is either disjoint from $c$ or else $s(\alpha^2)(c)=c$. We will see that $s(\alpha^2)(c) = c$ must hold. Let $S_A \subset S^2$ denote the subsurface bounded by $c$ that contains $A$.

We claim that the pair of surfaces $S_A$ and $s(\alpha^2)(S_A)$ satisfy the hypotheses of Lemma \ref{lemma:partition}. Each surface has a single boundary component $c, s(\alpha^2)(c)$ respectively, and we have already established that $c$ and $s(\alpha^2)(c)$ are either disjoint or equal. Each surface contains the point $A$, and as they are conjugate within $\Gamma$, each contains the same number of punctures. 

It remains to show that $s(\alpha^2)(c)$ is not contained in the interior of $S_A$. If this is the case, then $s(\alpha^2)(c)$ encloses a {\em strict} subset of the punctures contained in $S_A$. The curve $c$ induces a partition $P = P_1 \cup P_2$ of the set of punctures, and likewise $s(\alpha^2)(c)$ induces the conjugate partition $s(\alpha^2)(P)$. Without loss of generality, assume that $P_1$ corresponds to the punctures in $S_A$ and hence contains $A$, so that $P_2$ contains the points $x_3, \dots, x_n$. Since $s(\alpha^2)(c)$ encloses a strict subset of the punctures contained in $S_A$, one of the parts of $s(\alpha^2)(P)$ must be a strict subset of $P_1$. This part cannot be $s(\alpha^2)(P_1)$, since $s(\alpha^2)(P_1)$ has the same cardinality as $P_1$. But this part cannot be $s(\alpha^2)(P_2)$ either, since $P_2$ contains $x_3$ and hence $s(\alpha^2)(P_2)$ contains $x_5 \in P_2$.

By Lemma \ref{lemma:partition}, we have $\alpha^2(S_A) = S_A$. As $n \ge6$, also $\alpha^3 \sigma_1 \alpha^{-3} = \sigma_4$. The same argument then shows that $s(\alpha^3)(S_A) = S_A$, and hence $s(\alpha)(S_A) = S_A$. The claim follows.\end{proof}

We have shown that $s(\sigma_1)$ and $s(\alpha)$ both fix $c$ as well as the subsurfaces on either side of $c$. Let $S$ be the side containing the points $x_3, \dots, x_n$. In the case $\alpha = \alpha_0$, necessarily $S$ is globally invariant, in contradiction with Lemma \ref{lemma:reduciblesame}. If $\alpha = \alpha_1$, then we have shown that the image of the subgroup $\Mod_{n-1,1}(S^2) = \pair{\sigma_1, \alpha_1}$ under $s$ is contained in the subgroup $\Mod_{n-1,1,m}(S^2, S)$ of mapping classes fixing $S$. Composing with the map $r: \Mod_{n-1,1,m}(S^2, S) \to \Mod_{n-1,1,m',1}(S^2)$ obtained by restriction to $S$, we can now conclude the argument exactly as in the preceding Case A.2.a.
\end{enumerate}

\section{Proof of Theorem \ref{theorem:main}: Case B}\label{section:B}
The assumption in Case B is that the subsurfaces $S_3, \dots, S_n$ are all distinct. This case follows by an analysis of the boundary components of the subsurfaces $S_i$. A first observation is that the (necessarily disjoint) subsurfaces $S_i$ are all conjugate within $\Gamma$: for any $i \ge 3$, there is some $g \in \Mod_n(S^2)$ commuting with $\sigma_1$ and taking $x_3$ to $x_i$. Then $s(g)(S_3) = S_i$.

\begin{lemma}\label{lemma:c3}
For each $i = 3, \dots, n$, there is a unique component $c_i$ of $\partial S_i$ that separates $S_i$ from $p > \frac{n+m}{2}$ punctures.
\end{lemma}
\begin{proof}
Since the subsurfaces $S_i$ are all conjugate within $\Gamma$, it suffices to consider only $S_3$. Certainly if $c_3$ exists it must be unique. To see that it exists, denote the boundary components of $S_3$ by $d_1, \dots, d_k$. For $1 \le i \le k$, let $D_i$ denote the disk bounded by $d_i$ not containing $S_3$, and let $n_i$ denote the number of punctures in $D_i$. Without loss of generality, assume that $S_3$ is separated from $S_4 = s(\sigma_3)(S_3)$ by $d_1$. Then $S_4$ is separated from $S_3$ by some other element $s(\sigma_3)(d_i) \in \mathscr C$, for some $1 \le i \le k$. Necessarily $i = 1$, since if $i > 1$, then there is a strict containment $D_i \subset s(\sigma_3)(D_i)$, an absurdity. It follows that there is a containment
\[
s(\sigma_3)(D_2) \cup \dots \cup s(\sigma_3)(D_k) \cup s(\sigma_3(S_3)) \subset D_1,
\]
and hence, letting $n_0$ denote the number of punctures contained in $S_3$ itself,
\[
n_1 \ge n_2 + \dots + n_k + n_0.
\]
On the other hand, $\sum_{i = 0}^k n_i = n+m$, from which the inequality $n_1 \ge \frac{n+m}{2}$ follows. Moreover, this inequality must be strict, since otherwise the disks $D_1, s(\sigma_3)(D_1), s(\sigma_4 \sigma_3)(D_1)$ would be mutually disjoint and each would contain $\frac{n+m}{2}$ points. Taking $c_3 := d_1$, the result follows. 
\end{proof}

\begin{lemma}\label{lemma:sconj}
$s(\alpha_0^2)(c_3)=c_5$ and $s(\alpha_0^3)(c_3)=c_6$.
\end{lemma}
\begin{proof}
We define the {\em inside} of each $c_i$ to be the component $\int(c_i)$ that contains $x_i$, and define the {\em outside} as the other component. Each $\int(c_i)$ contains $q$ punctures, with $q < \frac{n+m}{2}$ by Lemma \ref{lemma:c3}. Define $c_5':= s(\alpha_0^2)(c_3)$. We again define the {\em inside} of $c_5'$ as the component $\int(c_5')$ containing $x_5$, and the {\em outside} as the other component. 

We claim that $\int(c_5)$ and $\int(c_5')$ satisfy the hypotheses of Lemma \ref{lemma:partition}. As $c_5 \in \mathscr C$ and $c_5' \in s(\alpha_0)^2 \mathscr C = \CRS(s(\sigma_3))$, we have that $c_5$ and $c_5'$ are either disjoint or equal. By definition, each contains $x_5$. The curves $c_3$ and $c_5$ contain the same number of punctures on their interiors, hence the same is true of $c_5$ and $c_5'$. It remains to be seen that $c_5'$ is not contained in the interior of $\int(c_5)$. If this is the case, then either the inside or the outside of $c_5'$ contains strictly fewer than $q$ punctures. But as the inside of $c_5'$ contains $q$ punctures and the outside contains $n+m-q > q$ punctures, this cannot be the case.

Applying Lemma \ref{lemma:partition}, it follows that $c_5 = c_5' = s(\alpha_0)^2(c_3)$ as claimed. Similar arguments establish the other claim. \end{proof}

Define the curves $c_1=s(\alpha_0^{-2})(c_3)$ and $c_2=s(\alpha_0^{-2})(c_4)$. 
\begin{lemma}
The curves $c_1, \dots, c_n$ are pairwise distinct and disjoint.
\end{lemma}
\begin{proof}
The curves $c_3, \dots, c_n$ are distinct and disjoint since they are all elements of $\mathscr C$ and each $c_i$ is distinguished by the property that it contains $x_i$ on its inside. The curves $c_1$ and $c_2$ are elements of $s(\alpha_0^{-2})\mathscr C = \CRS(\sigma_{n-1})$, and hence either disjoint from or equal to any element of $\mathscr C$. But $c_1$ and $c_2$ are uniquely characterized by the property of containing $x_1$ and $x_2$ in their interiors, respectively, and the claim follows. 
\end{proof}

\begin{lemma}\label{lemma:invariant}
The set $\{c_1,c_2,...,c_n\}$ is invariant under $\Gamma$. 
\end{lemma}
\begin{proof}
By Lemma \ref{lemma:genset}, $\Gamma$ is generated by the set $\{s(\sigma_1), s(\alpha_0)\}$. Thus it suffices to show that these two elements both preserve $\{c_1, \dots, c_n\}$. For $i = 3, \dots, n$, the element $s(\sigma_1)$ preserves each $S_i$, and hence also preserves the distinguished boundary component $c_i$. We claim that $s(\sigma_1)(c_1) = c_2$ and that $s(\sigma_1)(c_2) = c_1$. 

To see this, observe that $s(\sigma_3)(c_3) = c_4$. Thus $s(\alpha_0^2 \sigma_1 \alpha_0^{-2})(c_3) = c_4$, and so
\[
s(\alpha_0^2) s(\sigma_1)(s(\alpha_0^{-2})(c_3)) = c_4 = s(\alpha_0^2)(c_2).
\]
It follows that $s(\sigma_1)(s(\alpha_0^{-2})(c_3)) = s(\sigma_1)(c_1) = c_2$ as claimed. As also $s(\sigma_3)(c_4) = c_3$, the same reasoning shows that $s(\sigma_1)(c_2) = c_1$. 

It remains to see that the set $\{c_1, \dots, c_n\}$ is $\alpha_0$-invariant. We claim that $\alpha_0(c_i) = c_{i+1}$, interpreting subscripts mod $n$. It follows directly from Lemma \ref{lemma:sconj} that $s(\alpha_0)(c_5) = c_6$. As $\alpha_0 \sigma_4 = \sigma_5 \alpha_0$, 
\[
s(\alpha_0)(c_4) = s(\alpha_0 \sigma_4)(c_5) = s(\sigma_5)s(\alpha_0)(c_5) = s(\sigma_5)(c_6) = c_5,
\]
since $s(\sigma_i)(c_{i+1}) = c_i$ for $i = 3, \dots, n-1$. Then similar arguments show that $s(\alpha_0)(c_3) = c_4$, and $s(\alpha_0)(c_i) = c_{i+1}$ for $i = 6, \dots, n-1$. 

There are three remaining claims to establish:
\[
s(\alpha_0)(c_1) =c_2, \quad s(\alpha_0)(c_2) = c_3, \quad s(\alpha_0)(c_n) = c_1.
\]
Since $c_2:= s(\alpha_0^{-2})(c_4)$, the equality $s(\alpha_0)(c_2) = c_3$ follows from the above. Then the equality $s(\alpha_0)(c_1) = c_2$ follows by the same logic. Lastly, as $\alpha_0^n = \id$, 
\[
s(\alpha_0)(c_n) = s(\alpha_0^{1-n})(c_n) = c_1
\]
by what we have shown before. 
\end{proof}
Lemma \ref{lemma:invariant} and Lemma \ref{lemma:reducibledistinct} combine to show that $m$ must be divisible by $n(n-1)(n-2)$, contrary to assumption. Case B, and hence Theorem \ref{theorem:main}, follows.

    	\bibliography{sphericalbraid}{}
	\bibliographystyle{alpha}

\end{document}